\documentclass[reqno]{amsart}
\usepackage[colorlinks=true, pdfstartview=FitV, linkcolor=blue, citecolor=blue]{hyperref}
\usepackage{amssymb,amsmath,amsfonts}
\usepackage{bm}
\usepackage{enumitem}
\usepackage{array}
\allowdisplaybreaks

\newcommand{\N}{\mathbb{N}}
\newcommand{\R}{\mathbb{R}}
\newcommand{\C}{\mathbb{C}}
\newcommand{\Z}{\mathbb{Z}}
\newcommand{\Q}{\mathbb{Q}}
\newcommand{\rar}{\rightarrow}
\newcommand{\mc}{\mathcal}
\newcommand{\mr}{\mathrm}
\newcommand{\balph}{\bm{\alpha}}
\newcommand{\bmm}{\bm{m}}
\newcommand{\bk}{\bm{k}}
\newcommand{\btau}{\bm{\tau}}
\newcommand{\bt}{\bm{t}}
\newcommand{\bx}{\bm{x}}
\newcommand{\by}{\bm{y}}
\newcommand{\bv}{\bm{v}}
\newcommand{\be}{\bm{e}}
\newcommand{\cL}{\mathcal{L}}
\newcommand{\cU}{\mathcal{U}}
\newcommand{\cV}{\mathcal{V}}
\newcommand{\cA}{\mathcal{A}}
\def\SL{\operatorname{SL}}
\def\diag{\operatorname{diag}}
\def\Tr{\operatorname{Tr}}
\def\Nm{\operatorname{N}}
\def\Re{\operatorname{Re}}
\def\Im{\operatorname{Im}}

\newtheorem*{theorem*}{Theorem}
\newtheorem{theorem}{Theorem}
\newtheorem{lemma}[theorem]{Lemma}

\newtheorem{proposition}[theorem]{Proposition}

\begin{document}

\title[Gaps theorems for linear forms with cubic coefficients]{Gaps theorems for linear forms with cubic coefficients}
\author{Alan Haynes}


\begin{abstract}
Let $1,\alpha_1,\alpha_2$ be a basis for a cubic number field $K\subset\R$, and for $\tau_1,\tau_2\ge 1$ let $G(\balph,\btau)$ denote the number of distinct gaps between the fractional parts of the numbers $m_1\alpha_1+m_2\alpha_2$, with $m_1,m_2\in\Z$, $0\le m_1<\tau_1$ and $0\le m_2<\tau_2$. We show that $G(\balph,\btau)$ is unbounded as $\btau$ varies. When $K$ is totally real this is deduced from a density theorem of Shapira for orbits of the diagonal group on the space of lattices. When $K$ has a complex embedding we give an explicit construction, which also shows that there are constants $c,\kappa>0$ for which $G(\balph,\btau)\ge c(\log\max\{\tau_1,\tau_2\})^{\kappa}$ for infinitely many $\btau\in\N^2$.
\end{abstract}

\maketitle

\section{Introduction}\label{sec.Intro}

The three gap theorem, first proved by S\'os \cite{Sos1958}, Sur\'anyi \cite{Sura1958}, and \'Swierczkowski \cite{Swie1959}, states that for any $\alpha\in\R$ and any $N\in\N$, the points $\alpha,2\alpha,\ldots ,N\alpha$ partition $\R/\Z$ into intervals of at most three distinct lengths. In \cite{HaynMark2020} the author and Marklof studied the following higher dimensional version of this problem. For $\balph=(\alpha_1,\ldots ,\alpha_d)\in\R^d$ and for a bounded convex set $\mc D\subset\R^d$, let
\[
S(\balph,\mc D)=\{\bmm\cdot\balph\bmod 1:\bmm\in\Z^d\cap\mc D\}\subset\R/\Z,
\]
and let $G(\balph,\mc D)$ be the number of distinct lengths of the intervals into which the points of $S(\balph,\mc D)$ partition $\R/\Z$. For $d=1$ the three gap theorem says that $G(\alpha,\mc D)\le 3$ for every interval $\mc D$. For $d\ge 2$ the situation is different, and the main result of \cite{HaynMark2020} is that for almost every $\balph\in\R^d$, with respect to Lebesgue measure, $\sup_{R\ge 1}G(\balph,R\mc D)=\infty$ for every bounded convex $\mc D$ with non-empty interior. Here $R\mc D$ denotes the homothetic dilation of $\mc D$ by a factor of $R$.

A question which goes back to Erd\H{o}s (see the discussion at the end of \cite{GeelSimp1993}) asked whether $\sup_{R\ge 1}G(\balph,R[0,1)^2)=\infty$ whenever $1,\alpha_1,\alpha_2$ are linearly independent over $\Q$. The answer to this question is negative. It was observed by Boshernitzan and Dyson \cite{Bosh1991,Bosh1992,Dyson} that if $\balph$ is badly approximable then $\sup_{R\ge 1}G(\balph,R\mc D)<\infty$ for every bounded convex $\mc D$ with non-empty interior. A proof of this result was later published by Bleher, Homma, Ji, Roeder, and Shen \cite{BlehHommJiRoedShen2012}. In \cite[Theorem 2]{HaynMark2020} it is also shown to be an immediate consequence of a dynamical interpretation of $G(\balph,\mc D)$ together with Dani's correspondence.

In this paper we are interested in what happens for specific algebraic $\balph$, when the homothetic dilations $R\mc D$ are replaced by the more general dilations
\[
\mc D_{\btau}=\{(\tau_1x_1,\ldots ,\tau_dx_d):\bx\in\mc D\},\qquad \btau=(\tau_1,\ldots ,\tau_d),\ \tau_i\ge 1 .
\]
We take $d=2$ and $\mc D=[0,1)^2$ throughout, and we write
\[
S(\balph,\btau)=S(\balph,\mc D_{\btau})=\{m_1\alpha_1+m_2\alpha_2\bmod 1:\bmm\in\Z^2,\ 0\le m_1<\tau_1,\ 0\le m_2<\tau_2\},
\]
and $G(\balph,\btau)=G(\balph,\mc D_{\btau})$. We also write $|\btau|=\max\{\tau_1,\tau_2\}$. The case we have in mind is that in which $1,\alpha_1,\alpha_2$ is a $\Q$-basis for a cubic number field $K\subset\R$. Such an $\balph$ is badly approximable \cite{Perr1921} so, taking $\mc D=[0,1)\times[0,c)$ in the result of Boshernitzan and Dyson mentioned above, for each fixed $c>0$ the quantity $G(\balph,(R,cR))$ is bounded as a function of $R\ge 1$. Therefore if $G(\balph,\btau)$ is to be unbounded then this can only happen along sequences of $\btau$ which do not lie on a single ray. First we have the following result, for cubic fields with no complex embeddings.
\begin{theorem}\label{thm.TotReal}
Let $1,\alpha_1,\alpha_2$ be a $\Q$-basis for a totally real cubic field $K\subset\R$. Then
\[
\sup_{\tau_1,\tau_2\ge 1}G(\balph,\btau)=\infty .
\]
\end{theorem}
For cubic fields possessing a complex embedding we are able to prove more, namely that the number of gaps tends to infinity at a prescribed rate along a sequence of $\btau$ tending to $\infty$.
\begin{theorem}\label{thm.Complex}
Let $1,\alpha_1,\alpha_2$ be a $\Q$-basis for a cubic field $K\subset\R$ which has a complex embedding. Then there are constants $c,\kappa>0$, depending only on $\balph$, such that
\[
G(\balph,\btau)\ \ge\ c\,(\log|\btau|)^{\kappa}
\]
for infinitely many $\btau\in\N^2$. In particular $\sup_{\tau_1,\tau_2\ge 1}G(\balph,\btau)=\infty$.
\end{theorem}
The constants $c$ and $\kappa$ in Theorem \ref{thm.Complex} are effectively computable, but $\kappa$ comes from a lower bound for linear forms in logarithms, due to Baker and W\"{u}stholz, and is very small. We do not claim that the exponent is near optimal, and in fact numerical experiments with $\balph=(2^{1/3},4^{1/3})$ suggest that along the sequence of $\btau$ which our proof produces the number of gaps grows like a power of $\log|\btau|$ with exponent close to $1/2$.

The proof of Theorem \ref{thm.TotReal} uses the correspondence from \cite{HaynMark2020}, building on work from \cite{MarkStro2017}, which identifies the gaps in $S(\balph,\btau)$ with values of a certain function on the space of unimodular lattices in $\R^3$. This is combined with a theorem of Shapira \cite{Shap2011}, which says that the orbit of the relevant lattice under the full diagonal group $A$ is dense in that space. The gaps problem only sees the dilations with $\tau_1,\tau_2\ge 1$, which form a sub-semigroup of $A$, and density of an $A$-orbit does not in general give density of the orbit of a semigroup. For the lattices in question, however, this follows from an elementary observation about their dual lattices, and the proof of Theorem \ref{thm.TotReal} then becomes quite short.

For Theorem \ref{thm.Complex} we proceed instead by a direct construction. The unit group of the field has rank one, and if $u\in(0,1)$ is a unit and $\sigma$ is a complex embedding of the field then the powers $\sigma(u)^k$ spiral outward, with $\arg\sigma(u)$ an irrational multiple of $\pi$. Reading the coefficients of $u^k$ with respect to the basis $1,\alpha_1,\alpha_2$ from its complex embedding, we are able to find elements of $\Z+\Z\alpha_1+\Z\alpha_2$ whose coefficient vectors are of prescribed sign and size, and from two such elements we build a `ladder' of gaps of distinct lengths in $S(\balph,\btau)$. The same mechanism was used in \cite{DhanHayn2025} to study accumulation points of normalized approximations. The rate in Theorem \ref{thm.Complex} then reduces to the question of how long the sequence $k\arg\sigma(u)$ takes to enter a prescribed arc, and this is where Baker's theorem enters.

This paper is organized as follows. In Section \ref{sec.Lattices} we recall the correspondence between gaps and lattices from \cite{HaynMark2020}, in the coordinates we will use, together with the one result from that paper which we need. In Section \ref{sec.TotReal} we prove Theorem \ref{thm.TotReal}. Sections \ref{sec.Prelim}--\ref{sec.Complex} are devoted to the proof of Theorem \ref{thm.Complex}: Section \ref{sec.Prelim} collects the facts we need about cubic fields with a complex embedding, Section \ref{sec.Ladder} contains the combinatorial lemma which produces many gaps from two well chosen elements of the field, and Section \ref{sec.Complex} completes the proof.

\noindent\textbf{Acknowledgments:} The author would like to thank Jens Marklof for helpful conversations, and for introducing him to the lattice perspective for studying gaps problems. AI was used in this research, and to assist in the preparation of this manuscript. However, the final details of all proofs, as well as every statement, citation, and claim in the paper, has been checked by the author.

\noindent\textbf{Notation:} We use $O$ and $\ll$ to denote the standard big-oh and Vinogradov notations, with implied constants depending only on $\balph$ unless otherwise indicated. For $x\in\R$ we write $\|x\|_{\R/\Z}$ for the distance from $x$ to the nearest integer, and $\langle x\rangle$ for the unique element of $(x+\Z)\cap(0,1]$. Vectors in $\R^2$ and $\R^3$ are columns, written in boldface; matrices act on the left. For $\bx\in\R^2$ we write $\bx\cdot\balph=x_1\alpha_1+x_2\alpha_2$. Coordinates on $\R^3$ are written $\bx=(x_0,x_1,x_2)$. Finally, $\log$ denotes the natural logarithm.

\section{The gap counting function on the space of lattices}\label{sec.Lattices}

Let $X=\SL(3,\R)/\SL(3,\Z)$. We identify a point $\cL\in X$ with the unimodular lattice $\cL\Z^3\subset\R^3$ spanned by the columns of any representative of $\cL$, and for $\bt=(t_1,t_2)\in[0,1)^2$ we define
\begin{equation}\label{eqn.FDef}
F(\cL,\bt)=\min\bigl\{x_0>0:\bx\in\cL\Z^3,\ -t_i\le x_i<1-t_i\ (i=1,2)\bigr\},
\end{equation}
with the convention that $F(\cL,\bt)=\infty$ if no such $\bx$ exists. For $\balph\in\R^2$ and for $\btau$ with $\tau_1,\tau_2\ge 1$ we write
\[
u_{\balph}=\begin{pmatrix}1&0&0\\ \alpha_1&1&0\\ \alpha_2&0&1\end{pmatrix},\qquad
\hat g_{\btau}=\diag\bigl(\tau_1\tau_2,\tau_1^{-1},\tau_2^{-1}\bigr),\qquad
z_{\balph}=u_{\balph}^{\,t}\,\SL(3,\Z)\in X,
\]
and $W_{\btau}=\{\bmm\in\Z^2:0\le m_1<\tau_1,\ 0\le m_2<\tau_2\}$, so that $S(\balph,\btau)=\{\bmm\cdot\balph\bmod 1:\bmm\in W_{\btau}\}$. For $\bmm_0\in W_{\btau}$ we write $s(\bmm_0)$ for the length of the interval of $(\R/\Z)\setminus S(\balph,\btau)$ whose left endpoint is $\bmm_0\cdot\balph$, and, with $g_{\btau}=\diag(\tau_1,\tau_2)$, we write $g_{\btau}^{-1}\bmm_0=(m_{0,1}/\tau_1,m_{0,2}/\tau_2)\in[0,1)^2$. We will use two results from \cite{HaynMark2020}.
\begin{lemma}[{\cite[Section 2]{HaynMark2020}}]\label{lem.Dict}
Let $\balph\in\R^2$ and let $\tau_1,\tau_2\ge 1$. Then for every $\bmm_0\in W_{\btau}$,
\begin{equation}\label{eqn.Dict}
F\bigl(\hat g_{\btau} z_{\balph},\,g_{\btau}^{-1}\bmm_0\bigr)=\tau_1\tau_2\,s(\bmm_0).
\end{equation}
Consequently $G(\balph,\btau)$ is the number of distinct values taken by the left hand side of \eqref{eqn.Dict}, as $\bmm_0$ runs over $W_{\btau}$.
\end{lemma}
Lemma \ref{lem.Dict} is the identity $s_{k,B}=(\det B)^{-1}F(A_B,kB^{-1})$ of \cite[eq. (2.8)]{HaynMark2020}, in the special case $\mc D=[0,1)^2$, $B=g_{\btau}$ and $k=\bmm_0$, and we refer to \cite[Section 2]{HaynMark2020} for its proof. The only differences are in notation. In \cite{HaynMark2020} lattices are written as $\Z^3M$, with $M\in\SL(3,\R)$ acting on row vectors, the distinguished coordinate is the last one, and $F$ is a function on $\SL(3,\Z)\backslash\SL(3,\R)$. We also use Proposition 5 of \cite{HaynMark2020}, which gives the following result.
\begin{proposition}[{\cite[Proposition 5]{HaynMark2020}}]\label{prop.OpenSet}
There are constants $\lambda,\epsilon_0>0$ with the following property. For every $\epsilon\in(0,\epsilon_0]$ there are a non-empty open set $\cU\subseteq X$ and non-empty open sets $\cV_j\subseteq(0,1)^2$, $1\le j\le\lfloor\lambda/\epsilon\rfloor$, such that
\[
F(\cL,\bt)=\epsilon\,(j+a_j(\cL))\qquad\text{for all }\cL\in\cU\text{ and }\bt\in\cV_j,
\]
with $a_j(\cL)\in[-\tfrac{1}{100},\tfrac{1}{100}]$.
\end{proposition}

\section{Totally real cubic fields: Proof of Theorem \ref{thm.TotReal}}\label{sec.TotReal}

Let $A$ denote the group of diagonal matrices in $\SL(3,\R)$ with positive entries, acting on $X$ by left multiplication. The map $\btau\mapsto\hat g_{\btau}$, extended to all $\btau\in\R_{>0}^2$ by the same formula, is an isomorphism from $(\R_{>0}^2,\times)$ onto $A$, and the dilations which occur in the gaps problem correspond to the sub-semigroup $\{\hat g_{\btau}:\tau_1,\tau_2\ge 1\}$. The input we need from the dynamics of $A$ is the following theorem of Shapira, which he proved as part of his resolution of a problem of Cassels on the inhomogeneous minima of products of linear forms.

\begin{theorem*}[{\cite[Corollary 5.2]{Shap2011}}]
Let $1,\alpha_1,\alpha_2$ be a $\Q$-basis for a totally real cubic field $K\subset\R$. Then $\overline{Az_{\balph}}=X$.
\end{theorem*}

Density of the $A$-orbit of a lattice does not by itself give density of its orbit under the semigroup $\{\hat g_{\btau}:\tau_1,\tau_2\ge 1\}$, which is what the gaps problem sees. For the lattices $z_{\balph}$, however, the passage from the group to the semigroup is elementary, as we will now show. For $\cL\in X$ we write $\cL^*$ for the Euclidean dual lattice of $\cL\Z^3$, that is, the set of $\by\in\R^3$ with $\by\cdot\bx\in\Z$ for all $\bx\in\cL\Z^3$, so that $(M\Z^3)^*=M^{-t}\Z^3$ for $M\in\SL(3,\R)$.
We also write $\be_1=(0,1,0)^t$ and $\be_2=(0,0,1)^t$.

\begin{lemma}\label{lem.Tails}
Let $\balph\in\R^2$ be such that $\overline{Az_{\balph}}=X$. Then for every $T>0$ the set $\{\hat g_{\btau} z_{\balph}:\tau_1,\tau_2>T\}$ is dense in $X$.
\end{lemma}

\begin{proof}
The matrix $(u_{\balph}^{\,t})^{-t}=u_{\balph}^{-1}$ has second and third columns $\be_1$ and $\be_2$, so for every $\btau\in\R_{>0}^2$ we have that
\begin{equation}\label{eqn.DualVecs}
\tau_1\be_1,\ \tau_2\be_2\in\hat g_{\btau}^{-1}u_{\balph}^{-1}\Z^3=(\hat g_{\btau} z_{\balph})^* .
\end{equation}
Fix $T>0$ and let
\[
E_T=\bigl\{\cL\in X:s\be_i\in\cL^*\text{ for some }i\in\{1,2\}\text{ and some }0<s\le T\bigr\}.
\]
By \eqref{eqn.DualVecs}, $\hat g_{\btau} z_{\balph}\in E_T$ whenever $\min\{\tau_1,\tau_2\}\le T$. We claim that $E_T$ is closed and has empty interior. Granting this, let $\cU\subseteq X$ be non-empty and open. Then $\cU\setminus E_T$ is non-empty and open, so by hypothesis it contains a point $\hat g_{\btau} z_{\balph}$, and this point has $\tau_1,\tau_2>T$.

To see that $E_T$ is closed, let $\cL_n\in E_T$ converge to $\cL\in X$, with $s_n\be_i\in\cL_n^*$ and $0<s_n\le T$. Passing to a subsequence, we may assume that $i$ is the same for every $n$ and that $s_n\rar s\in[0,T]$. Choose representatives $M_n\rar M$ in $\SL(3,\R)$ of $\cL_n$ and $\cL$, and write $s_n\be_i=M_n^{-t}\bv_n$ with $\bv_n\in\Z^3\setminus\{\bm 0\}$. The vectors $\bv_n=s_nM_n^{\,t}\be_i$ are bounded, so after passing to a further subsequence $\bv_n=\bv$ is constant, and then $s_n\be_i=M_n^{-t}\bv\rar M^{-t}\bv$. Since $\bv\not=\bm 0$ we have that $M^{-t}\bv\not=\bm 0$, so $0<s\le T$ and $s\be_i=M^{-t}\bv\in\cL^*$. Hence $\cL\in E_T$.

To see that $E_T$ has empty interior, let $\cL\in X$ and let $k\in\mathrm{SO}(3)$. Then $(k\cL)^*=k\cL^*$, so $k\cL\in E_T$ only if $k\bv\in\R\be_1\cup\R\be_2$ for some $\bv\in\cL^*\setminus\{\bm 0\}$. For each such $\bv$ the set of $k\in\mathrm{SO}(3)$ with $k\bv\in\R\be_1\cup\R\be_2$ is a union of four cosets of $\mathrm{Stab}(\bv)\cong \mathrm{SO}(2)$, and so it has Haar measure zero. Since $\cL^*$ is countable, the set of $k$ with $k\cL\in E_T$ has measure zero, and in particular there are $k$ arbitrarily close to the identity with $k\cL\not\in E_T$. Since $k\cL\rar\cL$ as $k\rar e$, this shows that $\cL$ is not an interior point of $E_T$.
\end{proof}

\noindent\emph{Proof of Theorem \ref{thm.TotReal}:} Let $\lambda,\epsilon_0$ be as in Proposition \ref{prop.OpenSet}, let $\epsilon\in(0,\epsilon_0]$, and let $\cU$ and $\cV_j$, $1\le j\le\lfloor\lambda/\epsilon\rfloor$, be the open sets guaranteed by that proposition. Choose $\delta\in(0,1]$ such that each $\cV_j$ contains an open square of side length $\delta$, and put $T=2/\delta$. By Lemma \ref{lem.Tails}, together with the theorem preceding it, the set $\{\hat g_{\btau} z_{\balph}:\tau_1,\tau_2>T\}$ is dense in $X$, so we may choose $\btau$ with $\tau_1,\tau_2>T$ and $\hat g_{\btau} z_{\balph}\in\cU$. For each $j$ the set $\{(\tau_1t_1,\tau_2t_2):\bt\in\cV_j\}$ contains an open rectangle with sides of length greater than $2$, and therefore a point $\bmm_j\in\Z^2$; since $\cV_j\subseteq[0,1)^2$ we have that $\bmm_j\in W_{\btau}$, and by construction $g_{\btau}^{-1}\bmm_j\in\cV_j$. By Lemma \ref{lem.Dict} and Proposition \ref{prop.OpenSet},
\[
\tau_1\tau_2\,s(\bmm_j)=F\bigl(\hat g_{\btau} z_{\balph},\,g_{\btau}^{-1}\bmm_j\bigr)=\epsilon\,\bigl(j+a_j(\hat g_{\btau} z_{\balph})\bigr),
\]
and the numbers on the right hand side are distinct for distinct $j$, because $|a_j|\le 1/100$. It follows that $G(\balph,\btau)\ge\lfloor\lambda/\epsilon\rfloor$. Since $\epsilon\in(0,\epsilon_0]$ was arbitrary, this completes the proof of Theorem \ref{thm.TotReal}.

%We remark that the argument gives no information about the size of the $\btau$ which it produces: the density statement in Shapira's theorem is not effective, and the set $\cU$ of Proposition \ref{prop.OpenSet} shrinks as $\epsilon\rar 0$.

\section{Cubic fields with a complex embedding}\label{sec.Prelim}

For the remainder of the paper $K\subset\R$ is a cubic field with a complex embedding, $1,\alpha_1,\alpha_2$ is a $\Q$-basis for $K$, and
\[
\Lambda=\Z+\Z\alpha_1+\Z\alpha_2 .
\]
For $\beta\in\Lambda$ we write $\beta=n+m_1\alpha_1+m_2\alpha_2$ with $n\in\Z$ and $\bmm=\bmm(\beta)\in\Z^2$, and we refer to $\bmm(\beta)$ as the \textit{coefficient vector} of $\beta$ (which determines $\beta$ modulo $\Z$).

The real embedding of $K$ is implicitly fixed by the inclusion $K\subset\R$, and we write $\sigma$ for one of the two complex embeddings, so that
\[
\Tr(\beta)=\beta+2\Re\sigma(\beta),\qquad \Nm(\beta)=\beta\,|\sigma(\beta)|^2\qquad(\beta\in K).
\]
We will repeatedly use the fact that if $\gamma\in K$ and $\sigma(\gamma)\in\R$ then $\gamma\in\Q$: in that case $\sigma(\gamma)=\overline{\sigma(\gamma)}$, so $\gamma$ has at most two distinct conjugates, and $[\Q(\gamma):\Q]\le 2$ divides $[K:\Q]=3$. Finally, with $\alpha_0=1$, we write $\alpha_0^*,\alpha_1^*,\alpha_2^*$ for the dual basis of $\alpha_0,\alpha_1,\alpha_2$ with respect to the trace form, so that $\Tr(\alpha_i\alpha_j^*)=1$ if $i=j$, and $0$ otherwise.

\begin{lemma}\label{lem.Coeff}
For $\beta\in\Lambda$ with coefficient vector $\bmm$, and for $j=1,2$,
\begin{equation}\label{eqn.Coeff}
m_j=\Tr(\beta\alpha_j^*)=\alpha_j^*\beta+2\Re\bigl(\sigma(\alpha_j^*)\sigma(\beta)\bigr).
\end{equation}
Furthermore, $\sigma(\alpha_1^*)$ and $\sigma(\alpha_2^*)$ are non-zero, and their arguments are distinct modulo $\pi$.
\end{lemma}

\begin{proof}
Identity \eqref{eqn.Coeff} follows from the defining property of the dual basis. For the second assertion, $\sigma$ is injective, and if the arguments agreed modulo $\pi$ then $\sigma(\alpha_2^*/\alpha_1^*)$ would be real, so that $\alpha_2^*/\alpha_1^*\in\Q$, contradicting the linear independence of $\alpha_1^*$ and $\alpha_2^*$ over $\Q$.
\end{proof}

For $j=1,2$ put $P_j=|\sigma(\alpha_j^*)|>0$ and $\chi_j=\arg\sigma(\alpha_j^*)$. Replacing $\sigma$ by $\overline\sigma$ replaces each $\chi_j$ by $-\chi_j$, so in view of Lemma \ref{lem.Coeff} we may, and do, choose $\sigma$ so that
\[
\theta_0=\chi_2-\chi_1\in(0,\pi),\qquad \eta_0=\sin\theta_0\in(0,1],
\]
after adjusting $\chi_2$ by a multiple of $2\pi$.

\begin{lemma}\label{lem.Unit}
There is $u\in K$ with $0<u<1$ and $u\Lambda=\Lambda$. For any such $u$ we have that $|\sigma(u)|=u^{-1/2}$, and the number
\[
\vartheta=\frac{\arg\sigma(u)}{2\pi}
\]
is irrational. Consequently, for every $k_0\in\N$, the set $\{k\vartheta\bmod 1:k\ge k_0\}$ is dense in $\R/\Z$.
\end{lemma}

\begin{proof}
Let $Z_\Lambda=\{x\in K:x\Lambda\subseteq\Lambda\}$. This is a subring of $K$ which contains $\Z$ and is contained in $\Lambda$ (because $1\in\Lambda$), so it is a finitely generated $\Z$-module. Since $\mc O_K\Lambda$ is a finitely generated $\Z$-module which spans $K$ over $\Q$, there is an integer $c\ge 1$ with $c\,\mc O_K\Lambda\subseteq\Lambda$, and therefore $c\,\mc O_K\subseteq Z_\Lambda$. Thus $Z_\Lambda$ is an order in $K$, and by Dirichlet's unit theorem its unit group has rank $1$. Let $u_0\in Z_\Lambda^\times$ have infinite order. Since $u_0$ is real and is not $\pm 1$, we have that $|u_0|\not=1$, and one of the four numbers $\pm u_0^{\pm 1}$ lies in $(0,1)$; call it $u$. Then $u,u^{-1}\in Z_\Lambda$, so $u\Lambda=\Lambda$.

Now let $u$ be any such element. Then $u\in Z_\Lambda^\times$, so $\Nm(u)=\pm 1$, and since $\Nm(u)=u|\sigma(u)|^2>0$ we have that $|\sigma(u)|^2=1/u$. If $\vartheta$ were rational, say $\vartheta=p/q$ with $q\in\N$, then $\sigma(u^{2q})=|\sigma(u)|^{2q}e^{4\pi i p}$ would be real, so $u^{2q}\in\Q$. However, this would imply that $u=\pm 1$, a contradiction.
\end{proof}

We fix such a $u$ for the rest of the paper (for definiteness, the largest one), so that $u$ is determined by $\balph$ and constants depending on $u$ depend only on $\balph$. The next lemma supplies the second ingredient of our construction, which is a sequence of elements $\gamma\in\Lambda$ which are small in absolute value, but whose complex embeddings are large, lie in a prescribed quadrant, and are nearly real.

\begin{lemma}\label{lem.Mult}
There are a constant $c_1>0$ and a sequence of elements $\gamma^{(1)},\gamma^{(2)},\ldots$ of $\Lambda$, say $\gamma^{(r)}=x_r+y_r\alpha_1+z_r\alpha_2$ with $x_r,y_r,z_r\in\Z$, such that, writing $Q_r=|\sigma(\gamma^{(r)})|$,
\begin{enumerate}[topsep=2pt,itemsep=4pt,parsep=4pt]
\item[(i)] $|\gamma^{(r)}|\le 1/2$ and $\gcd(y_r,z_r)=1$;
\item[(ii)] $\Re\sigma(\gamma^{(r)})<0<\Im\sigma(\gamma^{(r)})$;
\item[(iii)] $Q_r\rar\infty$ as $r\rar\infty$, and $Q_r\,\Im\sigma(\gamma^{(r)})\le c_1$ for every $r$.
\end{enumerate}
\end{lemma}

\begin{proof}
For $(y,z)\in\R^2$ put
\[
\mc I(y,z)=y\,\Im\sigma(\alpha_1)+z\,\Im\sigma(\alpha_2),\]
and
\[\mc T(y,z)=\tfrac12\bigl(y(\Tr(\alpha_1)-3\alpha_1)+z(\Tr(\alpha_2)-3\alpha_2)\bigr).
\]
If $\gamma=x+y\alpha_1+z\alpha_2$ with $x,y,z\in\Z$ then \[\mc I(y,z)=\Im\sigma(\gamma).\] Furthermore, since
\[\Re\sigma(\gamma)=\frac12(\Tr(\gamma)-\gamma)\quad \text{and}\quad \Tr(\gamma)=3x+y\Tr(\alpha_1)+z\Tr(\alpha_2),\]we also have that
\begin{equation}\label{eqn.ReSigma}
\mc T(y,z)=\Re\sigma(\gamma)-\gamma.
\end{equation}

The linear forms $\mc I$ and $\mc T$ are linearly independent. To see this, for real $y,z$ put $\nu=y\alpha_1+z\alpha_2\in\R$ and $\zeta=y\sigma(\alpha_1)+z\sigma(\alpha_2)\in\C$. Then $\mc T(y,z)=\Re\zeta-\nu$, so if $(y,z)$ annihilated both forms then the $\R$-linear map $\R^3\rar\R\times\C$ given by $(x,y,z)\mapsto(x+y\alpha_1+z\alpha_2,\ x+y\sigma(\alpha_1)+z\sigma(\alpha_2))$ would send $(0,y,z)$ and $(\nu,0,0)$ to the same point $(\nu,\nu)$. This map is an isomorphism, since its matrix is obtained from the matrix $(\sigma_i(\alpha_j))$ of embeddings, hence $y=z=0$.

Next, $\mc I(y,z)\not=0$ for $(y,z)\in\Z^2\setminus\{(0,0)\}$, since otherwise $\sigma(y\alpha_1+z\alpha_2)$ would be real, so that $y\alpha_1+z\alpha_2\in\Q$, contradicting the linear independence of $1,\alpha_1,\alpha_2$. In particular $\Im\sigma(\alpha_2)\not=0$, and
\[
w=-\frac{\Im\sigma(\alpha_1)}{\Im\sigma(\alpha_2)}\qquad\text{satisfies}\qquad \mc I(y,z)=\Im\sigma(\alpha_2)\,(z-wy),
\]
with $w$ irrational. By linear independence $\mc T(1,w)\not=0$, and we write $\mc T(1,w)=-\delta c_0$ with $\delta\in\{\pm 1\}$ and $c_0>0$.

Let $p_r/q_r$ be the convergents of the continued fraction expansion of $w$, so that $q_r\rar\infty$, $\gcd(p_r,q_r)=1$, $0<|q_rw-p_r|<1/q_r$, and $q_rw-p_r$ alternates in sign. For the infinitely many $r$ for which the sign of $p_r-q_rw$ equals $\delta$ times the sign of $\Im\sigma(\alpha_2)$, put $(y,z)=\delta(q_r,p_r)$. Then $\gcd(y,z)=1$, $y=\delta q_r$, and
\[
\mc I(y,z)=\delta\,\Im\sigma(\alpha_2)\,(p_r-q_rw)\in\Bigl(0,\frac{|\Im\sigma(\alpha_2)|}{q_r}\Bigr).
\]
Since $(y,z)=y(1,w)+(0,z-wy)$ we have that
\[
\mc T(y,z)=y\,\mc T(1,w)+(z-wy)\,\mc T(0,1)=-c_0q_r+O(1/q_r).
\]
Choose $x\in\Z$ to be a nearest integer to $-(y\alpha_1+z\alpha_2)$, and let $\gamma=x+y\alpha_1+z\alpha_2$, so that $|\gamma|\le 1/2$. By \eqref{eqn.ReSigma}, $\Re\sigma(\gamma)\le\frac12-c_0q_r+O(1/q_r)$, which is negative for $r$ large, and together with $\Im\sigma(\gamma)=\mc I(y,z)>0$ this gives (ii). Finally $Q=|\sigma(\gamma)|\ge|\Re\sigma(\gamma)|\rar\infty$, while
\[
Q\le|\Re\sigma(\gamma)|+|\Im\sigma(\gamma)|\le c_0q_r+1
\]
for $r$ large, so that $Q\,\Im\sigma(\gamma)\le(c_0q_r+1)|\Im\sigma(\alpha_2)|/q_r\le 2c_0|\Im\sigma(\alpha_2)|$ once $q_r\ge 1/c_0$. Discarding finitely many $r$ and relabelling gives the sequence in the statement, with $c_1=2c_0|\Im\sigma(\alpha_2)|$.
\end{proof}

For $\gamma\in\Lambda$ as in Lemma \ref{lem.Mult} we write $Q=|\sigma(\gamma)|$ and
\begin{equation}\label{eqn.MuDef}
\sigma(\gamma)=-Q\,e^{i\mu},\qquad \mu\in(-\pi/2,0).
\end{equation}
Since $|\sin\mu|\ge 2|\mu|/\pi$ for $|\mu|\le\pi/2$, we have that
\begin{equation}\label{eqn.MuBd}
0<|\mu|\le\frac{\pi}{2}\cdot\frac{\Im\sigma(\gamma)}{Q},
\end{equation}
and in particular $|\mu|\le\pi c_1/(2Q^2)$ for the elements of the sequence in Lemma \ref{lem.Mult}.

\section{A ladder of gaps}\label{sec.Ladder}

In this section we show how suitably chosen elements of $\Lambda$ can be used to produce many distinct gaps in $S(\balph,\btau)$. Suppose that $\beta,\beta'\in\Lambda$ have coefficient vectors $\bmm,\bmm'$ with $m_1,m_2<0<m_1',m_2'$, that $m_1'$ and $m_2'$ are large compared with $|m_1|$ and $|m_2|$, respectively, that $\beta>0$ is very small, and that $|\beta'|<\beta$. Taking $\btau=(m_1',m_2')$, the vector $j\bmm+\bmm'=(m_1'-j|m_1|,\,m_2'-j|m_2|)$ lies in $W_{\btau}$ for every integer $1\le j\le\min\{m_1'/|m_1|,m_2'/|m_2|\}$ and, since $\beta$ is very small, for these $j$ the fractional parts of $(j\bmm+\bmm')\cdot\balph$ are the numbers $j\beta+\beta'$, an increasing `ladder' of small positive numbers, and we refer to the individual numbers $j\beta+\beta'$ as its `rungs'.

We will show that, by moving the base point $\bk_0$ in the formula
\[s(\bk_0)=\min_{\bk\in W_{\btau}-\bk_0}\langle\bk\cdot\balph\rangle,\]
which follows directly from the definition of $s(\bk_0)$, each rung of the ladder can be made to appear as a gap. Note that, since $1,\alpha_1,\alpha_2$ are linearly independent over $\Q$, the points $\bk\cdot\balph\bmod 1$ with $\bk\in W_{\btau}$ are distinct.

\begin{proposition}\label{prop.Ladder}
Let $\beta=n+\bmm\cdot\balph$ and $\beta'=n'+\bmm'\cdot\balph$ be elements of $\Lambda$, and suppose that
\begin{enumerate}[topsep=2pt,itemsep=4pt,parsep=4pt]
\item[(i)] $\beta>0$ and $m_1,m_2<0$;
\item[(ii)] $m_1',m_2'>0$ and, writing $M=m_1'$, $N=m_2'$, $C_1=M/|m_1|$ and $C_2=N/|m_2|$, we have that $C_1,C_2\ge 3$;
\item[(iii)] $|\beta'|<\beta$;
\item[(iv)] $4MN\beta<1$;
\item[(v)] the three $2\times 2$ minors of the matrix $\left(\begin{smallmatrix}n&m_1&m_2\\ n'&m_1'&m_2'\end{smallmatrix}\right)$ have greatest common divisor $1$.
\end{enumerate}
Then $G(\balph,(M,N))\ge\min\{\lfloor C_1\rfloor,\lfloor C_2\rfloor\}$.
\end{proposition}

\begin{proof}
Write $a=1/C_1=|m_1|/M$ and $b=1/C_2=|m_2|/N$, so that $a,b\le 1/3$, and write $\lambda_0=MN\beta$ and $\lambda_0'=MN\beta'$, so that $|\lambda_0'|<\lambda_0$ by (iii), and $\lambda_0<1/4$ by (iv). Put $J=\min\{\lfloor C_1\rfloor,\lfloor C_2\rfloor\}\ge 3$ and $Y=J\lambda_0+\lambda_0'$. Then $Y>0$, and since $J\le 1/\max\{a,b\}$ we have that $Y\le\lambda_0(J+1)\le\lambda_0\bigl(1/\max\{a,b\}+1\bigr)$.

\noindent\emph{Lattice points below the top of the ladder:} We first show that every $(\tilde n,\bk)\in\Z\times\Z^2$ with
\begin{equation}\label{eqn.Forcing}
0<\tilde n+\bk\cdot\balph\le\frac{Y}{MN},\qquad |k_1|<M,\qquad |k_2|<N,
\end{equation}
is an integer combination of $(n,m_1,m_2)$ and $(n',m_1',m_2')$. Let $\tilde\beta=\tilde n+\bk\cdot\balph$ and let $D$ be the determinant of the integer matrix with rows $(n,m_1,m_2)$, $(n',m_1',m_2')$ and $(\tilde n,k_1,k_2)$. Adding $\alpha_1$ times the second column and $\alpha_2$ times the third column to the first replaces the first column by $(\beta,\beta',\tilde\beta)^t$ without changing $D$, and expanding along the last row then gives
\[
D=\tilde\beta\,(m_1m_2'-m_2m_1')-k_1(\beta m_2'-m_2\beta')+k_2(\beta m_1'-m_1\beta').
\]
Here
\[m_1m_2'-m_2m_1'=|m_2|M-|m_1|N=MN(b-a),\]
while
\[|\beta m_2'-m_2\beta'|\le\beta N+|m_2||\beta'|\quad \text{and} \quad|\beta m_1'-m_1\beta'|\le\beta M+|m_1||\beta'|.\]
Using \eqref{eqn.Forcing} and $|\lambda_0'|<\lambda_0$ we have that
\[
|D|< Y|b-a|+2\lambda_0+(a+b)|\lambda_0'|\le Y|b-a|+\lambda_0(2+a+b).
\]
Now $|b-a|\le\max\{a,b\}\le 1/3$, so by our bound for $Y$ and by (iv),
\[
|D|<\lambda_0\bigl(1+\max\{a,b\}+2+a+b\bigr)\le 4\lambda_0<1.
\]
Since $D$ is an integer, $D=0$, and $(\tilde n,\bk)$ lies in the $\Q$-span of the first two rows, which is two dimensional because by (v) the minors are not all zero. By the theory of the Smith normal form, the greatest common divisor of the $2\times 2$ minors of an integer $2\times 3$ matrix of rank two is the product of the elementary divisors of the subgroup of $\Z^3$ generated by its rows, so (v) says precisely that this subgroup is the intersection of its $\Q$-span with $\Z^3$. This proves the claim.

We also note for later use that
\begin{equation}\label{eqn.YSmall}
\frac{Y}{MN}=J\beta+\beta'<(C_1+1)\beta\le 2C_1\beta=\frac{2M\beta}{|m_1|}\le 2M\beta=\frac{2\lambda_0}{N}<1 .
\end{equation}

\noindent\emph{The rungs:} Fix $j_0$ with $1\le j_0\le J$ and put $\bk_0=(0,(j_0-1)|m_2|)$. Then $\bk_0\in W_{(M,N)}$, because $(j_0-1)|m_2|\le(C_2-1)|m_2|=N-|m_2|<N$. We will show that $s(\bk_0)=j_0\beta+\beta'$. Since the $J$ numbers $j_0\beta+\beta'$ are distinct, this will prove the proposition.

For $j\ge 1$ the vector $\bk=j\bmm+\bmm'$ has $k_1=M-j|m_1|$ and $k_2=N-j|m_2|$, so $\bk\in W_{(M,N)}-\bk_0$ if and only if $0\le M-j|m_1|<M$ and $0\le N-(j-j_0+1)|m_2|<N$, which is to say that
\begin{equation}\label{eqn.RungRange}
1\le j\le C_1\qquad\text{and}\qquad j_0\le j\le C_2+j_0-1 .
\end{equation}
Both conditions hold for $j=j_0$, and for this $j$ we have that $\bk\cdot\balph\equiv j_0\beta+\beta'\bmod 1$, with $0<\beta+\beta'\le j_0\beta+\beta'\le Y/MN<1$ by (iii) and \eqref{eqn.YSmall}. Hence $\langle\bk\cdot\balph\rangle=j_0\beta+\beta'$, and $s(\bk_0)\le j_0\beta+\beta'$.

On the other hand, suppose that $\bk\in W_{(M,N)}-\bk_0$ satisfies $\langle\bk\cdot\balph\rangle\le j_0\beta+\beta'$, and let $\tilde n\in\Z$ be such that $\tilde n+\bk\cdot\balph=\langle\bk\cdot\balph\rangle$. Then $0\le k_1<M$ and $|k_2|<N$. By our choice of $\tilde n$ and the inequality $j_0\beta+\beta'\le Y/MN$ from the previous paragraph, we have that \[0<\tilde n+\bk\cdot\balph\le Y/MN.\]
By the first part of the proof, there are $a_1,a_2\in\Z$ with $(\tilde n,\bk)=a_1(n,m_1,m_2)+a_2(n',m_1',m_2')$, and therefore
\[
k_1=a_2M-a_1|m_1|,\qquad \langle\bk\cdot\balph\rangle=a_1\beta+a_2\beta' .
\]
The condition $0\le k_1<M$ says that $a_1/C_1\le a_2<a_1/C_1+1$, so that $a_2=\lceil a_1/C_1\rceil$. If $a_1\le 0$ then $a_1/C_1\le a_2\le 0$, so $|a_2|\le|a_1|/C_1$ and
\[
a_1\beta+a_2\beta'\le-|a_1|\beta+|a_1|\beta/C_1\le 0,
\]
which contradicts $\langle\bk\cdot\balph\rangle>0$. Therefore $a_1\ge 1$ and $a_2\ge 1$. If $a_2\ge 2$ then $a_1>(a_2-1)C_1\ge(a_2-1)J$, so that $a_1\ge(a_2-1)J+1\ge J+a_2-1$, and
\[
a_1\beta+a_2\beta'-\frac{Y}{MN}=(a_1-J)\beta+(a_2-1)\beta'\ge(a_2-1)(\beta+\beta')>0,
\]
which contradicts $\langle\bk\cdot\balph\rangle\le Y/MN$. Therefore $a_2=1$ and $\bk=a_1\bmm+\bmm'$ with $a_1\ge 1$, and \eqref{eqn.RungRange} gives $a_1\ge j_0$, so that $\langle\bk\cdot\balph\rangle=a_1\beta+\beta'\ge j_0\beta+\beta'$. This shows that $s(\bk_0)=j_0\beta+\beta'$, and completes the proof.
\end{proof}

\section{Proof of Theorem \ref{thm.Complex}}\label{sec.Complex}

The elements of $\Lambda$ which we feed into Proposition \ref{prop.Ladder} are $\beta=u^k$ and $\beta'=\gamma u^k$, with $u$ as in Lemma \ref{lem.Unit}, $\gamma$ from the sequence in Lemma \ref{lem.Mult}, and $k$ a large integer. The point of this choice is that multiplying both elements simultaneously by $u^k$ does not affect hypotheses (iii) and (v) of Proposition \ref{prop.Ladder}, but it allows us to rotate the complex embeddings of the pair into a favourable position, by choosing $k$ so that $k\vartheta$ lies in a prescribed arc modulo $1$.

\begin{lemma}\label{lem.Pair}
Let $\gamma=x+y\alpha_1+z\alpha_2\in\Lambda$ satisfy (i) and (ii) of Lemma \ref{lem.Mult}, let $Q$ and $\mu$ be as in \eqref{eqn.MuDef}, let $k\in\N$, and put
\[
\beta=u^k,\qquad\beta'=\gamma u^k,\qquad t=u^k,\qquad\varphi=2\pi k\vartheta .
\]
Then $\beta,\beta'\in\Lambda$, their coefficient vectors $\bmm,\bmm'$ satisfy
\begin{align}
t^{1/2}m_j&=2P_j\cos(\varphi+\chi_j)+\alpha_j^*\,t^{3/2},\label{eqn.PairM}\\
t^{1/2}m_j'&=-2P_jQ\cos(\varphi+\mu+\chi_j)+\alpha_j^*\gamma\,t^{3/2},\label{eqn.PairMp}
\end{align}
for $j=1,2$, and the three $2\times 2$ minors of the matrix of coefficients of $\beta,\beta'$ with respect to $1,\alpha_1,\alpha_2$ have greatest common divisor $1$.
\end{lemma}

\begin{proof}
We have $\beta,\beta'\in u^k\Lambda=\Lambda$. By Lemma \ref{lem.Unit}, $|\sigma(\beta)|=t^{-1/2}$ and $\arg\sigma(\beta)=\varphi$, so that $|\sigma(\beta')|=Qt^{-1/2}$ and $\arg\sigma(\beta')=\varphi+\mu+\pi$ by \eqref{eqn.MuDef}. Substituting into \eqref{eqn.Coeff} and multiplying by $t^{1/2}$ gives \eqref{eqn.PairM} and \eqref{eqn.PairMp}.

For the last assertion, multiplication by $u^k$ is an automorphism of $\Lambda$, so in coordinates with respect to $1,\alpha_1,\alpha_2$ it is given by a matrix in $\mr{GL}(3,\Z)$. Such a matrix acts on the vector of $2\times 2$ minors of a $2\times 3$ integer matrix by another element of $\mr{GL}(3,\Z)$, and so preserves the greatest common divisor of the minors. The pair $(\beta,\beta')$ is the image of the pair $(1,\gamma)$, whose coefficient matrix $\left(\begin{smallmatrix}1&0&0\\ x&y&z\end{smallmatrix}\right)$ has minors $y,z$ and $0$, and $\gcd(y,z)=1$.
\end{proof}

Now put
\[
\delta_0=\frac{1}{128P_1P_2}.
\]
For $\gamma$ as in Lemma \ref{lem.Pair} with $Q^2\ge 2\delta_0$, let $v_0\in(0,\pi/6]$ be defined by $\sin v_0=\delta_0/Q^2$, and let
\begin{equation}\label{eqn.ArcDef}
\cA_\gamma=\Bigl\{\frac{\pi}{2}-\chi_1-\mu+v:v_0\le v\le 2v_0\Bigr\}\subseteq\R/2\pi\Z .
\end{equation}

\begin{proposition}\label{prop.Window}
There are constants $Q_1,c_2,c_3,c_4>0$ with the following property. Let $\gamma$ be an element of the sequence of Lemma \ref{lem.Mult} with $Q\ge Q_1$, and let $k\ge c_2\log Q$ be an integer with $2\pi k\vartheta\in\cA_\gamma$ modulo $2\pi$. Then $\beta=u^k$ and $\beta'=\gamma u^k$ satisfy hypotheses (i)--(v) of Proposition \ref{prop.Ladder}, with
\[
\min\{C_1,C_2\}\ge\frac{Q}{c_3}\qquad\text{and}\qquad \max\{M,N\}\le c_4\,Q\,u^{-k/2}.
\]
In particular, with $\btau=(M,N)$,
\begin{equation}\label{eqn.WindowOut}
G(\balph,\btau)\ge\frac{Q}{c_3}-1\qquad\text{and}\qquad\log|\btau|\le\frac{k}{2}\log(1/u)+\log(c_4Q).
\end{equation}
\end{proposition}

\begin{proof}
By Lemma \ref{lem.Mult}(iii) and \eqref{eqn.MuBd}, both $v_0$ and $|\mu|$ tend to $0$ as $Q\rar\infty$ along the sequence, so we may choose $Q_1$ so large that $Q^2\ge 2\delta_0$ and
\begin{equation}\label{eqn.SmallAngles}
2v_0+|\mu|\le\eta_0/2
\end{equation}
whenever $Q\ge Q_1$. Write $\varphi=\frac{\pi}{2}-\chi_1-\mu+v$ with $v_0\le v\le 2v_0$, and $s=v+|\mu|$, so that $0<v<s\le\eta_0/2\le 1/2$. Then $\varphi+\chi_1=\frac{\pi}{2}+s$, $\varphi+\mu+\chi_1=\frac{\pi}{2}+v$, $\varphi+\chi_2=\frac{\pi}{2}+\theta_0+s$ and $\varphi+\mu+\chi_2=\frac{\pi}{2}+\theta_0+v$. The main terms in \eqref{eqn.PairM} and \eqref{eqn.PairMp} are therefore
\begin{align*}
F_1&:=2P_1\cos(\varphi+\chi_1)=-2P_1\sin s,\\
F_1'&:=-2P_1Q\cos(\varphi+\mu+\chi_1)=2P_1Q\sin v,\\
F_2&:=2P_2\cos(\varphi+\chi_2)=-2P_2\sin(\theta_0+s),\\
F_2'&:=-2P_2Q\cos(\varphi+\mu+\chi_2)=2P_2Q\sin(\theta_0+v).
\end{align*}
Since $\sin$ is $1$-Lipschitz and $\sin\theta_0=\eta_0$, \eqref{eqn.SmallAngles} gives $\sin(\theta_0+s),\sin(\theta_0+v)\in[\eta_0/2,3\eta_0/2]$. Since $v_0\le v\le 2v_0\le\pi/3$ we also have $\sin v_0\le\sin v\le 2\sin v_0$. Consequently
\begin{equation}\label{eqn.MainTermSizes}
|F_1|\ge\frac{2P_1\delta_0}{Q^2},\qquad F_1'\ge\frac{2P_1\delta_0}{Q},\qquad |F_2|\ge P_2\eta_0,\qquad F_2'\ge P_2Q\eta_0 .
\end{equation}

\noindent\emph{The error terms:} By \eqref{eqn.PairM} and \eqref{eqn.PairMp}, and since $|\gamma|\le 1/2$, each of $t^{1/2}m_j$ and $t^{1/2}m_j'$ differs from the corresponding main term by at most $\max\{|\alpha_1^*|,|\alpha_2^*|\}\,u^{3k/2}$. In view of \eqref{eqn.MainTermSizes} there is a constant $c_2$ such that, if $Q\ge 2$ and $k\ge c_2\log Q$, this quantity is at most one tenth of each of $|F_1|,F_1',|F_2|,F_2'$. For such $k$ each of $t^{1/2}m_1$, $t^{1/2}m_2$, $t^{1/2}m_1'$, $t^{1/2}m_2'$ has the same sign as its main term, and lies within a factor of $[\tfrac{9}{10},\tfrac{11}{10}]$ of it. We assume from now on that $Q_1\ge 2$ and that $k\ge c_2\log Q$.

\noindent\emph{The hypotheses of Proposition \ref{prop.Ladder}:} Hypothesis (i) holds because $\beta=u^k>0$ and $F_1,F_2<0$, and $m_1',m_2'>0$ because $F_1',F_2'>0$. Hypothesis (iii) holds because $|\beta'|=|\gamma|u^k\le\beta/2$, and (v) is the last assertion of Lemma \ref{lem.Pair}. For (ii), we have that
\[
C_1=\frac{m_1'}{|m_1|}\ge\frac{9}{11}\cdot\frac{F_1'}{|F_1|}=\frac{9}{11}\cdot\frac{Q\sin v}{\sin s},\qquad
C_2=\frac{m_2'}{|m_2|}\ge\frac{9}{11}\cdot\frac{F_2'}{|F_2|}\ge\frac{9}{11}\cdot\frac{Q\eta_0/2}{3\eta_0/2}=\frac{3Q}{11}.
\]
Since $\sin s=\sin(v+|\mu|)\le\sin v+|\mu|$ and $\sin v\ge\sin v_0$,
\[
\frac{\sin v}{\sin s}\ge\frac{1}{1+|\mu|/\sin v}\ge\frac{1}{1+|\mu|/\sin v_0},
\]
and by \eqref{eqn.MuBd} and Lemma \ref{lem.Mult}(iii), $|\mu|/\sin v_0\le\frac{\pi c_1}{2Q^2}\cdot\frac{Q^2}{\delta_0}=64\pi P_1P_2c_1$. Therefore $\min\{C_1,C_2\}\ge Q/c_3$ with $c_3=\max\{\frac{11}{9}(1+64\pi P_1P_2c_1),\frac{11}{3}\}$, and (ii) holds provided that $Q_1\ge 3c_3$. For (iv), we have that
\begin{align*}
MN\beta=(t^{1/2}m_1')(t^{1/2}m_2')&\le\Bigl(\frac{11}{10}\Bigr)^2F_1'F_2'
\le\frac{121}{100}\cdot4P_1P_2Q^2\cdot2\sin v_0\cdot\frac{3\eta_0}{2}\\
&=\frac{121}{100}\cdot 12P_1P_2\eta_0\delta_0\le\frac{121\cdot 12}{100\cdot 128}<\frac14 .
\end{align*}

\noindent\emph{The size of $\btau$:} Finally,
\begin{align*}
N=m_2'&\le\frac{11}{10}F_2'\,t^{-1/2}\le\frac{11}{10}\cdot 2P_2Q\cdot\frac{3\eta_0}{2}\,u^{-k/2}\le 4P_2\eta_0\,Q\,u^{-k/2},\\
M=m_1'&\le\frac{11}{10}F_1'\,t^{-1/2}\le\frac{11}{10}\cdot 2P_1Q\cdot\frac{2\delta_0}{Q^2}\,u^{-k/2}\le 5P_1\delta_0\,u^{-k/2},
\end{align*}
so we may take $c_4=\max\{4P_2\eta_0,5P_1\delta_0\}$. The inequalities \eqref{eqn.WindowOut} now follow from Proposition \ref{prop.Ladder}.
\end{proof}

Proposition \ref{prop.Window} reduces Theorem \ref{thm.Complex} to the question of how long the sequence $k\vartheta\bmod 1$ takes to enter the image of the arc $\cA_\gamma$ under division by $2\pi$, an arc of $\R/\Z$ of length $v_0/2\pi\ge\delta_0/(2\pi Q^2)$. To obtain a rate we need two more lemmas.

\begin{lemma}\label{lem.Hitting}
Let $\vartheta$ be irrational and let $\Phi:\N\rar(0,1]$ be non-increasing with $\|q\vartheta\|_{\R/\Z}\ge\Phi(q)$ for every $q\in\N$. Let $I\subseteq\R/\Z$ be an arc of length $\ell\in(0,1)$, let $k_0\ge 0$ be an integer, and put $R=\lceil 1/\ell\rceil$. Then there is an integer $k$ with
\[
k_0<k\le k_0+R\bigl(\Phi(R)^{-1}+1\bigr)\qquad\text{and}\qquad k\vartheta\bmod 1\in I .
\]
\end{lemma}

\begin{proof}
Replacing $I$ by $I-k_0\vartheta$, we may assume that $k_0=0$. By Dirichlet's theorem there is an integer $q$ with $1\le q\le R$ and $\|q\vartheta\|_{\R/\Z}<1/R\le\ell$. Put $\delta=\|q\vartheta\|_{\R/\Z}>0$, and let $\varepsilon=q\vartheta-p$, where $p$ is the nearest integer to $q\vartheta$, so that $|\varepsilon|=\delta$ and $jq\vartheta\equiv j\varepsilon\bmod 1$ for every $j$. Suppose, with little loss of generality, that $\varepsilon>0$. Write $I=[a,a+\ell)\bmod 1$ with $a\in[0,1)$, and let $j\ge 1$ be the least integer with $j\varepsilon>a$. Then $(j-1)\varepsilon\le a$, so $a<j\varepsilon\le a+\delta<a+\ell$, and $j\varepsilon\bmod 1\in I$. Moreover $j\le a/\varepsilon+1<1/\delta+1$, so $k=jq$ satisfies $k\le R(\Phi(R)^{-1}+1)$, because $\delta\ge\Phi(q)\ge\Phi(R)$.
\end{proof}

The second lemma is the one place where we use a deep result, a lower bound from Baker's theory on linear forms in logarithms. We use it in the following form.

\begin{theorem*}
Let $\xi$ be an algebraic number with $|\xi|=1$ which is not a root of unity. Then there are effectively computable constants $c_\xi>0$ and $C_\xi\ge 1$, depending only on $\xi$, such that $|\xi^q-1|\ge c_\xi\,q^{-C_\xi}$ for every $q\in\N$.
\end{theorem*}

This follows from the theorem of Baker and W\"ustholz \cite{BakeWust1993}. Let $\log$ denote the principal branch of the logarithm, so that $\log(-1)=i\pi$, let $q\in\N$, let $p$ be the integer nearest to $q\log\xi/(2\pi i)$, and put
\[
\Lambda=q\log\xi-2p\log(-1).
\]
Then $\Lambda$ is purely imaginary with $|\Lambda|\le\pi$, and it is non-zero because $\xi$ is not a root of unity. Since $\xi^q=e^{\Lambda}$ and $|e^{ix}-1|=2|\sin(x/2)|\ge 2|x|/\pi$ for $|x|\le\pi$, we have that $|\xi^q-1|\ge 2|\Lambda|/\pi$. Finally, $|p|\le(q+1)/2$, so both coefficients of $\Lambda$ are at most $q+1$ in absolute value, and the theorem of Baker and W\"ustholz, applied to $\Lambda$ as a linear form in the logarithms of the algebraic numbers $\xi$ and $-1$, gives $|\Lambda|\ge(q+1)^{-C}$ for a constant $C$ depending only on $\xi$.

\begin{lemma}\label{lem.Baker}
There are effectively computable constants $c_5>0$ and $C\ge 1$ such that
\[
\|q\vartheta\|_{\R/\Z}\ge c_5\,q^{-C}\qquad\text{for every }q\in\N .
\]
\end{lemma}

\begin{proof}
Put $\xi=\sigma(u)/\overline{\sigma(u)}$. This is an algebraic number, being a quotient of two conjugates of $u$, it has $|\xi|=1$, and $\xi=e^{2i\arg\sigma(u)}=e^{4\pi i\vartheta}$. It is not a root of unity, since $\xi^n=1$ would force $2n\vartheta\in\Z$, contradicting Lemma \ref{lem.Unit}. By the quoted theorem,
\[
c_\xi q^{-C_\xi}\le|\xi^q-1|=2|\sin(2\pi q\vartheta)|\le 2\pi\|2q\vartheta\|_{\R/\Z}\le 4\pi\|q\vartheta\|_{\R/\Z},
\]
which gives the lemma with $c_5=c_\xi/4\pi$ and $C=C_\xi$.
\end{proof}

\noindent\emph{Proof of Theorem \ref{thm.Complex}:} Let $\gamma$ be an element of the sequence in Lemma \ref{lem.Mult} with $Q\ge Q_1$, let $\cA_\gamma$ be the arc \eqref{eqn.ArcDef}, and let $I\subseteq\R/\Z$ be the image of $\cA_\gamma$ under division by $2\pi$, an arc of length
\[
\ell=\frac{v_0}{2\pi}\ge\frac{\sin v_0}{2\pi}=\frac{\delta_0}{2\pi Q^2}.
\]
Apply Lemma \ref{lem.Hitting} with $k_0=\lceil c_2\log Q\rceil$, with $\Phi(q)=c_5q^{-C}$ from Lemma \ref{lem.Baker} (decreasing $c_5$ if necessary so that $\Phi\le 1$), and with $R=\lceil 1/\ell\rceil\le 2\pi Q^2/\delta_0+1$. This produces an integer $k\ge c_2\log Q$ with $2\pi k\vartheta\in\cA_\gamma$ modulo $2\pi$ and
\[
k\le c_2\log Q+1+R\bigl(c_5^{-1}R^{C}+1\bigr)\le c_6\,Q^{2(C+1)},
\]
for a constant $c_6$ depending only on $\balph$ and $u$. By Proposition \ref{prop.Window} there is $\btau=(M,N)\in\N^2$ with
\[
G(\balph,\btau)\ge\frac{Q}{c_3}-1\qquad\text{and}\qquad\log|\btau|\le\frac{k}{2}\log(1/u)+\log(c_4Q)\le c_7\,Q^{2(C+1)}.
\]
The second inequality gives $Q\ge(\log|\btau|/c_7)^{1/(2(C+1))}$, and substituting this into the first gives
\[
G(\balph,\btau)\ge c\,(\log|\btau|)^{\kappa},\qquad\kappa=\frac{1}{2(C+1)},
\]
for a suitable $c>0$, provided that $Q$ is sufficiently large. Letting $\gamma$ run through the sequence of Lemma \ref{lem.Mult}, we have $Q\rar\infty$, so $G(\balph,\btau)\rar\infty$. This completes the proof of Theorem \ref{thm.Complex}.

We close with a remark. Baker's theorem is used only for the rate, and it is possible to circumvent it, at the cost of a more slowly growing lower bound for the gaps function. Both $\sigma(u)$ and $\overline{\sigma(u)}$ are algebraic units, so the number $\xi$ in the proof of Lemma \ref{lem.Baker} is a unit in the normal closure of $K$, a field of degree at most $6$ over $\Q$. For $q\in\N$ the number $\xi^q-1$ is a non-zero algebraic integer whose conjugates are bounded in absolute value by $H^q+1$, where $H>1$ is the largest absolute value of a conjugate of $\xi$. Taking norms gives $|\xi^q-1|\ge(H^q+1)^{-5}$, and hence $\|q\vartheta\|_{\R/\Z}\gg H^{-5q}$ in place of Lemma \ref{lem.Baker}. The argument above then gives $G(\balph,\btau)\gg(\log\log|\btau|)^{1/2}$ for infinitely many $\btau$.

\vspace{.15in}
{\footnotesize
\noindent
Department of Mathematics\\
University of Houston\\
Houston, TX, United States\\
haynes@math.uh.edu
}

\end{document}